\documentclass{article}[a4paper]

\usepackage{lineno}
\usepackage{amsthm} 
\usepackage{amsmath} 
\usepackage{amssymb} 
\usepackage{mathrsfs}
\usepackage{graphicx} 
\usepackage{cite} 
\usepackage{url} 
\usepackage{setspace} 
\usepackage{geometry} 
\usepackage{hyperref} 
\hypersetup{hidelinks} 
\usepackage{todonotes} 
\usepackage{color} 
\usepackage{indentfirst}
\usepackage{cases}
\usepackage{bm}
\usepackage{float}
\usepackage{subcaption}
\usepackage[section]{placeins}
\usepackage{graphicx}
\numberwithin{equation}{section}

\newtheorem{thm}{Theorem}[section]
\newtheorem{lem}[thm]{Lemma}

\newtheorem{cor}[thm]{Corollary}
\newtheorem{prop}{Proposition}[section]
\newtheorem{rem}[thm]{Remark}

\begin{document}
	\title{Standing wave solutions of $abcd$-systems for water waves}

\author{Peifei Song\\
	School of Mathematical Sciences,\\
	University of Electronic Science and Technology of China,\\
	Chengdu, Sichuan 610054, China \\
	\texttt{songpf@std.uestc.edu.cn}	\and
	Yuhao Xie\\
	School of Mathematical Sciences,\\
	University of Electronic Science and Technology of China,\\
	Chengdu, Sichuan 610054, China \\
	\texttt{xieyuhao@std.uestc.edu.cn}	\and
	Min Chen\\
	Department of Mathematics,\\
Eastern Institute of Technology,\\  Ningbo,Zhejiang 315200, China\\ Purude University,\\ IN 47907, USA\\ 	\texttt{chen45@purdue.edu}	\and
	Shenghao Li\footnote{Corresponding author} \\
	School of Mathematical Sciences,\\
	University of Electronic Science and Technology of China,\\
	Chengdu, Sichuan 610054, China \\
	\texttt{lish@uestc.edu.cn}}
\date{}

\maketitle
\begin{abstract}
We continue the study for  standing wave solutions of $abcd$-systems which was started by Chen and Iooss \cite{chen2005standing} for the BBM system via the Lyapunov-Schmidt method. In this paper, we will first discuss the feasibility of the Lyapunov-Schmidt method for bifurcating standing wave solutions of  $abcd$-systems. These systems will be characterized into three categories: feasible, infeasible and uncertain feasible ones. In particular, we prove the existence of nontrivial bifurcating standing waves for the Bona-Smith system.
\end{abstract}



\section{Introduction}

 A  class of equations with four parameters, known as $abcd$-systems,
\begin{equation}
	\label{1}
	\left\{ \begin{array} {l}
		\eta_t+u_x+(u \eta)_x+au_{xxx}-b \eta_{x x t}=0,\\
		u_t+\eta_x+u u_x+c \eta_{x x x}-du_{x x t}=0,\\
	\end{array} \right.
\end{equation}
 was introduced by Bona, Chen and Saut \cite{Bona2002283,Bona2004925},  which describes the two-way propagation of small-amplitude, long wavelength of an incompressible, inviscid fluid in a uniform
horizontal channel of finite depth. System (\ref{1}) is the first-order approximations  of two-dimensional Euler equations for free-surface flow. The coordinate $ x $ which measures the distance along channel is scaled by $ h_0 $ and time $ t $ is scaled by $\sqrt{h_0 / g} $ , with $ h_0 $ being the average depth of water in its quiescent state and $ g $ being acceleration of gravity. The dependent variables $\eta(x, t)$ and $u(x, t)$, scaled by $ h_0 $ and $\sqrt{g h_0}$ respectively, represent the dimensionless deviation of the water surface from its undisturbed position and the horizontal velocity at the certain depth of the undisturbed fluid. The parameters $a$, $b$, $c$ and $d$ appearing in system (\ref{1}) are not independently specified, obeying the relations,
\begin{equation}
\begin{aligned}
	a&=  \frac{1}{2}\left(\theta^2-\frac{1}{3}\right) \lambda, & b & =\frac{1}{2}\left(\theta^2-\frac{1}{3}\right)(1-\lambda), \\
	c & =\frac{1}{2}\left(1-\theta^2\right) \mu,~ & d & =\frac{1}{2}\left(1-\theta^2\right)(1-\mu),
\end{aligned}
\end{equation}
where $\theta \in[0,1]$ and $\lambda, \mu \in \mathbb{R}$. Hence, this leads to,
\begin{equation}
\begin{aligned}
		\label{50}
	a+b & =\frac{1}{2}\left(\theta^2-\frac{1}{3}\right), \\
	c+d & =\frac{1}{2}\left(1-\theta^2\right) \geq 0, \\
	a+b+c+d & =\frac{1}{3}.
\end{aligned}
\end{equation}

According to   previous studies \cite{chen2005standing,CHEN2006393,li2019diamond,li2019standing} on the Coupled BBM-system $(a=c=0, b=d=\frac{1}{6})$, we notice that the Lyapunov-Schmidt method plays an important role in the study for the existence of bifurcating standing wave or traveling wave solutions. In addition, Chen and Iooss \cite{chen2005standing} point out that  the situation on the Coupled BBM-system is very different from the standing gravity waves problem for the full water wave equation solved by Plotnikov et al. in finite \cite{iooss2005standing} and infinite depth case \cite{plotnikov2001nash}. In the finite depth case \cite{iooss2005standing}, there is a delicate \emph{small divisor problem} in the inversion of the linearized operator. In the infinite depth case \cite{plotnikov2001nash}, there occurs an \emph{infinite dimensional  kernel problem} for the linearized operator at the critical value of the parameter, which presents the additional difficulty of complete resonance. One can avoid  these two problems while conduct studies on some $abcd-$systems, and,  based one Chen-Iooss's  note \cite{chen2005standing}, the Lyapunov-Schmidt method  might apply on the model system (1)  when
$abcd < 0$, or $b = c = d = 0,~a> 0$ (the Kaup system \cite{kaup1975higher,sachs1988integrable}), or when $a = b = c = 0,~d> 0$ (the so called classical Boussinesq
system \cite{boussinesq1871theorie,peregrine1972equations,wu1994bidirectional,schonbek1981existence}) or when $a = 0,~b = d > 0,~c< 0$  (such as the Bona-Smith system \cite{bona1976model}). 

In this article, we will first systematically study the feasibility of the Lyapunov-Schmidt method for bifurcating standing wave solutions on $abcd$-systems, based on the prerequisite of avoiding the following two problems:
\begin{itemize}
	\item Infinite kernel problem,
	\item Small divisor problem.
\end{itemize}
While we establish the desired bifurcation solutions, local well-posedness theories for the system is required. Thanks to the previous study   \cite{Bona2002283} that 
the Cauchy problem of linear $abcd$-systems
is well-posed  if and only if one of the following sets of conditions is satisfied:
$$
\begin{aligned}
&(C1)~a \leq 0, \quad c \leq 0, \quad b \geq 0, \quad d \geq 0;\\
&(C2)~a=c>0, \quad b \geq 0, \quad d \geq 0;\\
&(C3)~  a=c>0, \quad b=d<0.
\end{aligned}
$$
 And in the study of nonlinear $abcd$-systems in \cite{Bona2004925}, it has been shown that systems satisfying $(C1)$ and $(C2)$ are  well-posed. However, due to the limited practical applicability of systems satisfying $(C3)$, their well-posedness in the nonlinear case has not been extensively discussed. Therefore, the scope of our research is limited to systems satisfying conditions $(C1)$ and $(C2)$. Through the analysis, we categorize the systems with well-posedness into three classes under the conditions in (\ref{50}). The first class is the feasible systems including three types of $abcd$-systems as follows:
\begin{itemize}
	\item[$\bullet$] $a=0,~b>0,~c<0,~d>0$, which includes the Bona-Smith system with,
$$
a=0,~b=d=\frac{1-\mu}{3(2-\mu)}>0 ~\text { and }~ c=\frac{\mu}{3(2-\mu)}<0,
$$	
corresponding to the case where $\theta^2=\left(\frac{4}{3}-\mu\right) /(2-\mu),~\lambda=0,~\mu<0$;
	\item[$\bullet$]$a<0,~b>0,~c=0,~d>0$;
    \item[$\bullet$]$a=0,~b>0,~c=0,~d>0$, whcih includes the Coupled BBM-system with,
$$
a=c=0,~b=d=\frac{1}{6}.
$$
\end{itemize}
The second class is the infeasible systems including $2$ types $abcd$-systems, where the small divisor problem exists. The third class is the uncertain feasible systems including $11$ types $abcd$-systems, where the kernel space dimension might be finite for only a small subset of parameters. In addition, these problems may be overcome to a certain extent by KAM theory \cite{berti2019kam,chierchia2000kam
}, Nash-Moser theorem \cite{iooss2015small,plotnikov2001nash}, global bifurcation theory \cite{rabinowitz1971some} and other improved methods.

Next, our study will focus on the Bona-Smith system as follows, whose theories are quite different from the Coupled BBM-system (See. \cite{Bona2002283} and the  references therein),
\begin{equation}
	\label{2}
	\left\{ \begin{array} {l}
	 \eta_t+u_x+(u \eta)_x-b \eta_{x x t}=0,\\
	 u_t+\eta_x+u u_x+c \eta_{x x x}-b u_{x x t}=0.\\
\end{array} \right.
\end{equation}
The critical case ($\mu \rightarrow-\infty$) of (\ref{2}), where $
b=\frac{1}{3} \text { and } c=-\frac{1}{3}$, is introduced by Bona and Smith \cite{bona1976model}. We prove the existence of nontrivial bifurcating standing waves for Bona-Smith system. The bifurcation set is formed with the discrete union of Whitney's umbrellas in the three-dimensional space formed with 3 parameters representing the time-period, the wave length, and the average of wave amplitude.

The paper is organized as follows. In section 2, we discuss the feasibility of the Lyapunov-Schmidt method for bifurcating standing wave solutions on $abcd$-systems, based on the well-posedness results of the systems. In particular, we prove the existence of nontrivial bifurcating standing waves for Bona-Smith system in the later sections. In Section 3, we present the linearized problem and obtain its solution near zero. In Section 4, the Lyapunov-Schmidt method is applied to derive the bifurcation standing waves. In Section 5, we present an illustrative example of standing waves as predicted by the theorems.
\section{Research on feasibility}

 In this section, we will discuss the feasibility of the Lyapunov-Schmidt method for bifurcating standing wave solutions on $abcd$-systems. As pointed out in the introduction, whether the Lyapunov-Schmidt method can be applied to $abcd$-systems for bifurcating standing wave solutions depends on two problems: the \emph{infinite kernel problem} and the\emph{ small divisor problem}. Moreover, to obtain the bifurcation solutions, local well-posedness is required. Therefore, the scope of our research is limited to systems that satisfy conditions $(C1)$ and $(C2)$.

 Since we are looking for standing wave  solutions in $(x,t)$, let us first introduce the scaled variables $\tilde{x}=\frac{2 \pi}{\lambda} x$, $\tilde{t}=\frac{2 \pi}{T} t$, with $\lambda$ and $T$ being the wave length and time period. Then re-scaled $abcd$-systems of (\ref{1}) are obtained by dropping the tilde,
\begin{equation}
	\label{3}
	\left\{\begin{array}{l}
		\eta_t+\beta u_x+\alpha \beta a u_{xxx}-\alpha b \eta_{x x t}+\beta\left(u\eta\right)_x=0, \\
		u_t+\beta \eta_x+\alpha \beta c \eta_{x x x}-\alpha du_{x x t}+\beta u u_x=0,
	\end{array}\right.
\end{equation}
where $\alpha$ and $ \beta$ are positive parameters defined by,
$$
\alpha=(2 \pi)^2 / \lambda^2, \quad \beta=T / \lambda.
$$
The standing waves we are looking for are solutions $(\eta,u)$ doubly $2 \pi$-periodic functions of $(x,t)$. Then we write the linearized system of ($\ref{3}$) as follows, providing with $f$ given odd in $x$, $g$ given even in $x$,
\begin{equation}
	\label{4}
	\left\{ \begin{array} {l}
		\eta_t+\beta u_x+\alpha \beta a u_{xxx}-\alpha b \eta_{x x t}=f_x, \\
		u_t+\beta \eta_x+\alpha \beta c \eta_{x x x}-\alpha du_{x x t}=g_x,\\
	\end{array} \right.
\end{equation}
where we look for solutions $(\eta,u)$ with $\eta$ even in $x$ and $u$ odd in $x$. Based on the above properties of functions, we introduce the Sobolev spaces as follows,
$$
H_{\natural\natural}^k=H^k\left\{(\mathbb{R} / 2 \pi \mathbb{Z})^2\right\},
$$
$$
H_{\natural\natural}^{k, e}=\left\{w \in H_{\natural\natural}^k ; w \text { is even in } x\right\},H_{\natural\natural}^{k, o}=\left\{w \in H_{\natural\natural}^k ; w \text { is odd in } x\right\}.
$$In order to use a suitable linearized operator to describe the system ($\ref{4}$), we now define operator $\pi_0$ by,
\begin{equation}
	\label{5}
	\left(\pi_0 g\right)(t)=\frac{1}{2 \pi} \int_{-\pi}^\pi g(x, t)~ \mathrm{d} x,
\end{equation}
and $D_x^{-1}$ by the inverse of the differential operator, which suppresses the average firstly and then takes the primitive with an average of 0. In consideration of the evenness  of $\eta$ and the oddness of $u$, the linearized operator $\mathcal{L}$ is defined as follows,
$$
	\begin{aligned}
		\mathcal{L}(\eta, u) & =D_x^{-1}\left(u_t+\beta \eta_x+\alpha \beta c \eta_{x x x}-\alpha du_{x x t}, \eta_t+\beta u_x+\alpha \beta a u_{xxx}-\alpha b \eta_{x x t}\right) \\
		& =\left(D_x^{-1} u_t+\beta\left(\mathbb{I}-\pi_0\right) \eta+\alpha \beta c\left(\mathbb{I}-\pi_0\right)\eta_{xx}-\alpha d\left(\mathbb{I}-\pi_0\right)u_{x t}, D_x^{-1}\eta_t+\beta u+\alpha\beta au_{xx}-\alpha b \eta_{x t}\right).
	\end{aligned}
$$
The system (\ref{4}) is then written as,
\begin{equation}
	\label{7}
	\mathcal{L}U=F,
\end{equation}
where $U=(\eta,u), F=(\tilde{g},f)=(\left(\mathbb{I}-\pi_0\right) g,f)\in H_{\natural\natural}^{k, e} \times H_{\natural\natural}^{k, o}$, which is equivalent to $(g,f)\in H_{\natural\natural}^{k, e} \times H_{\natural\natural}^{k, o}, k\geqslant0$.

Let us write the Fourier series,
\begin{equation}
\begin{aligned}
	\label{47}
	& \eta(x, t)=\sum_{p \geqslant 0, p, q \in \mathbb{Z}} \eta_{p q}(\cos p x) \mathrm{e}^{\mathrm{i} q t}, ~
	& u(x, t)=\sum_{p>0, p, q \in Z} u_{p q}(\sin p x) \mathrm{e}^{\mathrm{i} q t},\\
	& f(x, t)=\sum_{p>0, p, q \in Z} f_{p q}(\sin p x) \mathrm{e}^{\mathrm{i} q t}, ~
	& g(x, t)=\sum_{p \geqslant 0, p, q \in Z} {g}_{p q}(\cos p x) \mathrm{e}^{\mathrm{i} q t}.
\end{aligned}
\end{equation}
Then, substituting the above Fourier series into (\ref{4}), we obtain that,
\begin{itemize}
	\item[$(i)$] for $p > 0$, $q \in \mathbb{Z}$, where $\tilde{g}_{p q}=g_{p q}$ can be checked, 
	\begin{equation}
		\label{8}
		\left\{\begin{array}{l}
			i q\left(1+\alpha b p^2\right) \eta_{p q}+\beta p\left(1-\alpha a p^2\right)u_{pq}=p f_{pq},  \\
			\beta p\left(1-\alpha c p^2\right) \eta_{p q}-i q\left(1+\alpha d p^2\right)u_{pq}=p \tilde{g}_{p q};
		\end{array}\right.
	\end{equation}
	\item[$(ii)$] for $p=0$, $q \in \mathbb{Z}$,
	$$
	\eta_{0 q}=0, ~\mathrm {when}~q \neq 0~\mathrm{and}~ \eta_{00}~\mathrm{is}~\mathrm{arbitrary}.
	$$
\end{itemize}
For system (\ref{8}), let us define, 
\begin{equation}
	\label{9}
	\Delta(p,q)=q^2\left(1+\alpha b p^2\right)\left(1+\alpha d p^2\right)-\beta^2 p^2\left(\alpha a p^2-1\right)\left(\alpha c p^2-1\right).
\end{equation}
If $\Delta(p,q) = 0$, the solutions to the equation $\Delta(p,q)=0$ correspond to the non-trivial kernel of $\mathcal{L}$. To ensure that the kernel space of the linearized operator $\mathcal{L}$ is finite-dimensional, it is necessary that the equation $\Delta(p,q)= 0$ has only a finite number of solutions in $ \mathbb{N} \times \mathbb{Z}$.  Hence, the discussion on $\Delta(p,q)$ is crucial for our study, as it is the key to exploring the \emph{infinite kernel problem}.\\
If $\Delta(p,q) \neq 0$, we can solve (\ref{8}) as follows,
\begin{equation}
	\label{10}
	\begin{aligned}
		\eta_{p q}=-\Delta^{-1} p\left[i q\left(1+\alpha dp^2\right) f_{p q}+ \beta p(1-\alpha ap^2)\tilde{g}_{p q}\right], 
	\end{aligned}
\end{equation}	
\begin{equation}
	\label{11}
	\begin{aligned}
		u_{p q}=-\Delta^{-1} p\left[\beta p\left(1-\alpha cp^2\right) f_{p q }-i q\left(1+\alpha bp^2\right) \tilde{g}_{p q}\right],
	\end{aligned}
\end{equation}
and then we denote $A$, $B$ in the following for later use,
\begin{equation}
\begin{aligned}
	\label{48}
	& A:=\frac{p\left|q\left(1+\alpha d p^2\right)\right|+\beta p^2 |\left(1-\alpha cp^2 \right)|}{| q^2\left(1+\alpha b p^2\right)\left(1+\alpha d p^2\right)-\beta^2 p^2\left(\alpha a p^2-1\right)\left(\alpha cp^2-1 \right)|},\\
	& B:=\frac{\beta p^2\left|\left(1-\alpha a p^2\right)\right|+p|q\left(1+\alpha b p^2\right)|}{| q^2\left(1+\alpha b p^2\right)\left(1+\alpha d p^2\right)-\beta^2 p^2\left(\alpha a p^2-1\right)\left(\alpha cp^2-1 \right)|}.
\end{aligned}
\end{equation}
To avoid  the \emph{small divisor problem}, we will need to give a bound for the pseudo-inverse of the linearized operator $\mathcal{L}$, which corresponds to giving estimates for $\left(\eta_{p q}, u_{p q}\right)$ in terms of $\left(f_{p q}, \tilde{g}_{p q}\right)$ for noncritical couples. Therefore, we will need to verify that $A$ and $B$ are bounded(independent of the choices for($p,q$)). 
\subsection{Examples of three cases}
Based on the above conditions, we will study three different cases (including feasible case, unfeasible case and uncertain feasible case) as examples to demonstrate the process of verifying the feasibility of the Lyapunov-Schmidt method. 
\vspace{0.1cm}\\{\bm{{Case \uppercase\expandafter{\romannumeral1} :~$a=0,~b>0,~c<0,~d>0$}}.
\begin{itemize}
	\item[$\bullet$] \textbf{Infinite kernel problem} 
\end{itemize}

	To verify the dimension of the kernel space is finite, the key is that the equation $\Delta(p,q)= 0$ has only a finite number of solutions in $\mathbb{N} \times \mathbb{Z}$.  Following from ($\ref{9}$), when $\Delta(p,q) = 0$, we have,
\begin{equation}
\label{49}
	q^2\left(1+\alpha b p^2\right)\left(1+\alpha d p^2\right)+\beta^2 p^2\left(\alpha c p^2-1\right)=0,
\end{equation}
and one has,
$$
q^2=-\frac{\beta^2 p^2\left(\alpha c p^2-1\right)}{\left(1+\alpha b p^2\right)\left(1+\alpha d p^2\right)}.
$$
We can define the sequence of number $\left\{a_p\right\}$ as follows,
$$
	a_p:=-\frac{\beta^2 p^2\left(\alpha c p^2-1\right)}{\left(1+\alpha b p^2\right)\left(1+\alpha d p^2\right)},~ p\in\mathbb{N}.
$$
Let $ p $ tend to $+\infty$, we have,
$$
\lim _{p \rightarrow+\infty} a_p=-\lim _{p \rightarrow +\infty} \frac{\beta^2 p^2\left(\alpha c p^2-1\right)}{\left(1+\alpha b p^2\right)\left(1+\alpha d p^2\right)}=-\frac{\beta^2c}{\alpha bd},
$$
 which represents the positive value number of $q^2$ is finite. Then for a fixed value of $q$, recalling from (\ref{49}),  $p$ is the solution of the following equation,
\begin{equation}
	\label{13}
	\left(\alpha^2 b d q^2-\alpha \beta^2 c\right) p^4+\left(\alpha b q^2+\alpha d q^2-\beta^2\right) p^2+q^2=0,
\end{equation}
where there are at most four possible values of $p$ such that (\ref{13}) holds. This finishes the proof that $\Delta(p,q)= 0$ has only finite solutions in $\mathbb{N} \times \mathbb{Z}$.
\begin{itemize}
	\item[$\bullet$]  \textbf{Small divisor problem}
\end{itemize}

 Recalling from (\ref{48}), when $a=0,~b>0,~c<0,~d>0$, we have,
$$
\begin{aligned}
	& A=\frac{p\left|q\left(1+\alpha d p^2\right)\right|+\beta p^2 |\left(1-\alpha cp^2 \right)|}{| q^2\left(1+\alpha b p^2\right)\left(1+\alpha d p^2\right)+\beta^2 p^2\left(\alpha cp^2-1 \right)|},\\
	& B=\frac{\beta p^2+p|q\left(1+\alpha b p^2\right)|}{| q^2\left(1+\alpha b p^2\right)\left(1+\alpha d p^2\right)+\beta^2 p^2\left(\alpha cp^2-1 \right)|}.
\end{aligned}
$$
By comparing the order of $p$ and $q$ in the numerator and denominator of $A$ and $B$, there exists a constant $M>0$ (depending only on ($\alpha, \beta$ and $b, c, d$)) such that for any pair $(p, q) \in \mathbb{N} \times \mathbb{Z}$ satisfying $\Delta(p, q)\neq0$ and $(p, q) \neq(0,0)$, we have,
$$
A\leqslant M,~~B\leqslant M.
$$
Hence, based on(\ref{10})-(\ref{11}), and by applying the triangle inequality, we obtain the following result,
$$
	\left|\eta_{p q}\right|+\left|u_{p q}\right| 
	\leqslant M(|f_{pq}|+|\tilde{g}_{p q}|),
$$
which can give estimates for $\left(\eta_{p q}, u_{p q}\right)$ in terms of $\left(f_{p q}, \tilde{g}_{p q}\right)$, and then it ensures the boundedness for the pseudo-inverse of the linearized operator $\mathcal{L}$ on its range.

Therefore, when $a=0,~b>0,~c<0,~d>0$, the $abcd$-systems can apply the Lyapunov-Schmidt method for bifurcating standing wave solutions.
\vspace{0.1cm}\\{\bm{{Case \uppercase\expandafter{\romannumeral2} :~$a=c=d=0,~b=\frac{1}{3}$}}.
\begin{itemize}
	\item[$\bullet$]  \textbf{Small divisor problem}
\end{itemize}

 Recalling from (\ref{48}), when $a=c=d=0,~b=\frac{1}{3}$, we have,
$$
\begin{aligned}
 A=\frac{p\left|q\right|+\beta p^2} {| q^2\left(1+\frac{1}{3}\alpha  p^2\right)-\beta^2 p^2|},~~
 B=\frac{\beta p^2+p|q\left(1+\frac{1}{3}\alpha p^2\right)|}{| q^2\left(1+\frac{1}{3}\alpha p^2\right)-\beta^2 p^2|}.
\end{aligned}
$$
For $B$, we can define the sequence of number $\left\{b_p\right\}$ for a fixed $q$ as follows,
$$
	b_p:=\frac{\beta p^2+p|q\left(1+\frac{1}{3}\alpha p^2\right)|}{| q^2\left(1+\frac{1}{3}\alpha p^2\right)-\beta^2 p^2|},~ p\in\mathbb{N}.
$$
Let $ p $ tend to $+\infty$, we have,
$$
 \lim _{p \rightarrow+\infty} b_p=\lim _{p \rightarrow+ \infty}\frac{\beta p^2+p|q\left(1+\frac{1}{3}\alpha p^2\right)|}{| q^2\left(1+\frac{1}{3}\alpha p^2\right)-\beta^2 p^2|}\rightarrow +\infty,
$$
which  represents the existence of a small divisor problem. 

Therefore, when $a=c=d=0,~b=\frac{1}{3}$, the Lyapunov-Schmidt method cannot be used for bifurcating standing wave solutions.
\vspace{0.1cm}\\{\bm{{Case \uppercase\expandafter{\romannumeral3} :~$a<0,~b>0,~c<0,~d>0$}}.
\begin{itemize}
	\item[$\bullet$] \textbf{Infinite kernel problem}
\end{itemize}

As in the case \uppercase\expandafter{\romannumeral1}, when $\Delta(p,q) = 0$,  the following results are obtained from (\ref{9}),
$$
q^2\left(1+\alpha b p^2\right)\left(1+\alpha d p^2\right)-\beta^2 p^2\left(\alpha a p^2-1\right)\left(\alpha c p^2-1\right)=0,
$$
and one has,
$$
q^2=\frac{\beta^2 p^2\left(\alpha a p^2-1\right)\left(\alpha c p^2-1\right)}{\left(1+\alpha b p^2\right)\left(1+\alpha d p^2\right)}.
$$
We can define the sequence of number $\left\{c_p\right\}$ as follows,
$$
c_p:=\frac{\beta^2 p^2\left(\alpha a p^2-1\right)\left(\alpha c p^2-1\right)}{\left(1+\alpha b p^2\right)\left(1+\alpha d p^2\right)},~ p\in\mathbb{N}.
$$
Let $ p $ tend to $+\infty$, we have,
$$
\lim _{p \rightarrow+\infty} c_p=\lim _{p \rightarrow+\infty} \frac{\beta^2 p^2\left(\alpha a p^2-1\right)\left(\alpha c p^2-1\right)}{\left(1+\alpha b p^2\right)\left(1+\alpha d p^2\right)}\rightarrow +\infty,
$$ 
which represents the positive value number of $q^2\in \mathbb{R}$ is infinite. Therefore, according to different parameters $\alpha, \beta$ and $a, b, c, d $, there will be the following two cases:
\begin{itemize}
	\item[$(1)$]  an infinite number of pairs  $(p, q) \in \mathbb{N} \times \mathbb{Z}$ satisfies $\Delta(p,q) = 0$ ; 
\end{itemize}
\begin{itemize}
	\item[$(2)$]  only a finite number of pairs  $(p, q) \in \mathbb{N} \times \mathbb{Z}$ satisfies $\Delta(p,q) = 0$.
\end{itemize}
 Only when the second case is satisfied, we can use the Lyapunov-Schmidt method for bifurcating standing wave solutions.
\begin{itemize}
	\item[$\bullet$]  \textbf{Small divisor problem}
\end{itemize}

Recalling from the process of Case \uppercase\expandafter{\romannumeral1}, it is easy to find that we can give estimates for $\left(\eta_{p q}, u_{p q}\right)$ in terms of $\left(f_{p q}, \tilde{g}_{p q}\right)$, and then allow to give a bound of the pseudo-inverse of the linearized operator $\mathcal{L}$ on its range. 

Therefore, when $a=0,~b>0,~c<0,~d>0$ , the applicability of the Lyapunov-Schmidt method  to the $abcd$-systems for bifurcating standing wave solutions depends on the parameters $\alpha, \beta$ and $a, b, c, d $.

\subsection{Conclusion}

Similar to the study process in Section 2.1,  regarding whether the Lyapunov-Schmidt method can be used for bifurcating standing wave solutions, we categorize the systems with
 well-posedness into three classes under condition ($\ref{50}$). The first class consists of feasible systems, including three types of $abcd$-systems. The second class consists of infeasible systems, including two types of $abcd$-systems, where the small divisor problem exists. The third class consists of uncertain feasible systems, including eleven types of $abcd$-systems; only when the parameters $\alpha, \beta$ and $a, b, c, d $  satisfy the condition that the kernel space is finite-dimensional can the Lyapunov-Schmidt method be used.

\begin{itemize}
	\item[$(i)$]The feasible systems
	\item[$\bullet$]$a=0,~b>0,~c<0,~d>0$, and it includes the Bona-Smith system with,
	$$
	a=0,~b=d=\frac{1-\mu}{3(2-\mu)}>0 ~ \text { and } ~c=\frac{\mu}{3(2-\mu)}<0,
	$$
	corresponding to the case where $\theta^2=\left(\frac{4}{3}-\mu\right) /(2-\mu),~\lambda=0,~\mu<0$;

	\item[$\bullet$]$a<0,~b>0,~c=0,~d>0$;
	\item[$\bullet$]$a=0,~b>0,~c=0,~d>0$, and it includes the Coupled BBM-system with,
	$$
	a=c=0,~b=d=\frac{1}{6}.
	$$

	\item[$(ii)$] The infeasible systems
	\item[$\bullet$]$a=c=d=0,~b=\frac{1}{3}$;
	\item[$\bullet$]$a=b=c=0,~d=\frac{1}{3}$,~which is called Classical Boussinesq system. 
	\item[$(iii)$] The uncertain feasible systems\\
	This class constitutes the  remaining eleven cases, including the Coupled KdV system with,
$$
a=c=\frac{1}{6},~b=d=0.
$$

\end{itemize}

\section{Linearized operator of Bona-Smith system}

In the subsequent sections of this paper, we will focus on the existence of nontrivial bifurcating standing waves for the Bona-Smith system. In Section 2, we only provide a brief proof of feasibility for 3 cases. A rigorous proof for the Bona-Smith system  will be presented later. Since we are looking for  standing wave solutions in $(x,t)$, let us now introduce the scaled variables,
$$\tilde{x}=\frac{2 \pi}{\lambda \sqrt{b}} x,\quad\tilde{t}=\frac{2 \pi }{T\sqrt{b}} t,$$with $\lambda \sqrt{b}$ and ${T\sqrt{b}}$ being the wave length and time period. The re-scaled Bona-Smith system is obtained by dropping the tilde,
\begin{equation}
	\label{15}
	\left\{ \begin{array} {l}
		\eta_t+\beta u_x-\alpha \eta_{x x t}+\beta(u \eta)_x=0, \\
		u_t+\beta \eta_x-\gamma \alpha \beta \eta_{x x x}-\alpha u_{x x t}+\beta u u_x=0,\\
	\end{array} \right.
\end{equation}
where $\alpha$, $ \beta$ and $\gamma$ are positive parameters defined by,
$$
\alpha=(2 \pi)^2 / \lambda^2, \quad \beta=T / \lambda,\quad \gamma =-c/b\quad(0 <\gamma<1)  .
$$
The standing waves we are looking for are solutions $(\eta,u)$ doubly $2 \pi$-periodic functions of $(x,t)$.

We start by studying the linearized system, providing with $f$ given odd in $x$, $g$ given even in $x$,
\begin{equation}
\left\{ \begin{array} {l}
	\eta_t+\beta u_x-\alpha \eta_{x x t}=f_x, \\
 	u_t+\beta \eta_x-\gamma \alpha \beta \eta_{x x x}-\alpha u_{x x t}=g_x,\\
\end{array} \right.
	\label{16}
\end{equation}
where we look for solutions $(\eta,u)$ with $\eta$ even in $x$, $u$ odd in $x$.
We now define the linearized operator $\mathcal{L}$ as follows,
\begin{equation}
	\label{17}
\begin{aligned}
\mathcal{L}(\eta, u) & =D_x^{-1}\left(u_t+\beta \eta_x-\gamma \alpha \beta \eta_{x x x}-\alpha u_{x x t}, \eta_t+\beta u_x-\alpha \eta_{x x t}\right) \\
& =\left(D_x^{-1} u_t+\beta\left(\mathbb{I}-\pi_0\right) \eta-\gamma \alpha \beta\left(\mathbb{I}-\pi_0\right)\eta_{xx}-\alpha\left(\mathbb{I}-\pi_0\right)u_{x t}, D_x^{-1}\eta_t+\beta u-\alpha \eta_{x t}\right),
\end{aligned}
\end{equation}
where $\pi_0$ is defined in ($\ref{5}$).
Then the system (\ref{16}) is written as,
\begin{equation}
	\label{18}
\mathcal{L}U=F,
\end{equation}
where $U=(\eta,u), F=(\tilde{g},f)=(\left(\mathbb{I}-\pi_0\right) g,f)\in H_{\natural\natural}^{k, e} \times H_{\natural\natural}^{k, o}$, which is equivalent to $(g,f)\in H_{\natural\natural}^{k, e} \times H_{\natural\natural}^{k, o}, k\geqslant0$.
By substituting the  Fourier series expansion of $\eta, u$ and $f ,g$  from (\ref{47}) into (\ref{16}), we can obtain that,
\begin{itemize}
	\item[$(i)$] for $p > 0$, $q \in \mathbb{Z}$, where $\tilde{g}_{p q}=g_{p q}$ can be checked, 
	\begin{equation}
		\label{19}
		\left\{\begin{array}{l}
			i q\left(1+\alpha p^2\right) \eta_{p q}+ \beta p  u_{p q}=p f_{p q}, \\
			\left( \beta p+\gamma \alpha \beta p^3 \right) \eta_{p q}-i q\left(1+\alpha p^2\right) u_{p q}=p \tilde{g}_{p q};
		\end{array}\right.
	\end{equation}
	\item[$(ii)$] for $p=0$, $q \in \mathbb{Z}$,
$$
\eta_{0 q}=0, ~\mathrm {when}~q \neq 0~\mathrm{and}~ \eta_{00}~\mathrm{is}~\mathrm{arbitrary}.
$$
\end{itemize}
In the rest of the paper, we look for $\eta$  in the corresponding invariant subspace,
$$
H_{\natural\natural, 0}^{k, e}=\left\{\eta \in H_{\natural\natural}^{k, e} ; \eta_{0 q}=0, q \neq 0\right\} .
$$
We now start to solve the system (\ref{19}), and let us define, 
\begin{equation}
	\label{20}
	\Delta(p,q)=q^2(1+\alpha p^2)^2-\beta^2p^2 \left(1+\gamma \alpha p^2\right).
\end{equation}
If $\Delta(p,q) \neq 0$, we get,
\begin{equation}
	\label{21}
\begin{aligned}
	 \eta_{p q}=-\Delta^{-1} p\left[i q\left(1+\alpha p^2\right) f_{p q}+ \beta p\tilde{g}_{p q}\right], \quad\quad\quad\quad\quad
\end{aligned}
\end{equation}	
\begin{equation}
	\label{22}
\begin{aligned}
	u_{p q}=-\Delta^{-1} p\left[\left( \beta p+\gamma  \alpha \beta p^3\right) f_{p q }-i q\left(1+\alpha p^2\right) \tilde{g}_{p q}\right].
\end{aligned}
\end{equation}
If $\Delta(p,q)=0$ and $q\neq 0$, according to the second equation of (\ref{19}) we have,
$$
u_{p q}=\frac{p \tilde{g}_{p q}-\left( \beta p+\gamma  \alpha \beta p^3\right) \eta_{p q}}{-i q\left(1+\alpha p^2\right)},
$$
then it follows from the first equation of (\ref{19}) that,
$$
\left(q^2\left(1+\alpha p^2\right)^2- \beta^2p^2\left(1+\gamma \alpha p^2\right)\right) \eta_{p q}=-i p q\left(1+\alpha p^2\right) f_{p q}- \beta p^2 \tilde{g}_{p q}.
$$
Therefore, based on $\Delta(p,q)=q^2(1+\alpha p^2)^2- \beta^2p^2(1+\gamma \alpha p^2)=0$, we have the compatibility condition for $ F=(\tilde{g}, f)$ in the form,
\begin{equation}
	\label{23}
	\begin{aligned}
		& \tilde{g}_{pq}+ sgn(q)i\sqrt{1+\gamma \alpha p^2} f_{pq}=0.
	\end{aligned}
\end{equation}
Hence, if there exists a pair of $(p,q) $ satisfying compatibility condition (\ref{23}), the system (\ref{19}) has solutions,
\begin{equation}
\label{24}
\begin{aligned}
	&\eta_{pq}=\frac{-sgn(q) i \sqrt{1+\gamma \alpha p^2} p f_{pq}}{q\left(1+\alpha p^2\right) \sqrt{1+\gamma\alpha p^2}+ \beta p}+C, \\
	&u_{pq}=\frac{pf_{pq}}{q\left(1+\alpha p^2\right) \sqrt{1+\gamma \alpha p^2}+ \beta p}-sgn(q) i C \sqrt{1+\gamma \alpha p^2}, 
\end{aligned}
\end{equation}
where $C$ is arbitrary in $\mathbb{C}$. 

It can be observed that the linearized operator $\mathcal{L}$ defined in ($\ref{18}$) has a kernel for $(p, q) =(0, 0)$ or for $(p, q)$ satisfying $\Delta(p, q) = 0$. The following lemmas provide detailed information on the values of  $(p, q)$ where a nontrivial kernel of the linearized operator $\mathcal{L}$ exists.
\begin{lem}
Given $\alpha, \beta>0$, $0<\gamma<1$, the set,
\begin{equation}
	\label{25}
\Sigma_{(\alpha, \beta,\gamma)}:=\left\{(p, q) \in \mathbb{N}^{2}, q\left(1+\alpha p^2\right)-\beta p  \sqrt{1+\gamma \alpha p^2}=0\right\},
\end{equation}
is either empty, or finite. When there exists $\left(p_0, q_0\right) \in \Sigma_{(\alpha, \beta, \gamma)}$, then $\left(p_0, q_0\right)$ is the only element of $\Sigma_{(\alpha, \beta,\gamma)}$ if the following two conditions are satisfied:

(i) $ \frac{q_0^2}{\left(\alpha^2 q_0^2-\gamma \alpha \beta^2\right) p_0^2}$ is not an integer;

(ii) $ \frac{\beta^2-2 \alpha q^2 \pm \sqrt{\left(2 \alpha q^2-\beta^2\right)^2-4\left(\alpha q^2-\gamma \alpha \beta^2\right) q^2}}{2\left(\alpha^2 q^2-\gamma \alpha \beta^2\right)}$ are not integers for $q=1,2, \ldots, L$ with $q \neq q_0$ and $L=[\frac{\beta \sqrt{1+\alpha}}{\alpha}]$.
\end{lem}
\begin{proof}
	For $\alpha, \beta>0$, $0<\gamma<1$, $(p, q) \in \Sigma_{(\alpha, \beta,\gamma)}$ implies,
$$
q=\frac{\beta p\sqrt{1+\gamma\alpha  p^2}}{1+\alpha p^2}\leq \frac{\beta \sqrt{1+\gamma\alpha}}{\alpha} \leq \frac{\beta \sqrt{1+\alpha}}{\alpha}.
$$
Hence, the only possible values for $q$ are $q=1,2, \ldots,[\frac{\beta \sqrt{1+\alpha}}{\alpha}] $. Then for a fixed value of $q$, recalling from (\ref{20}), $p$ is the solution of the following equation,
\begin{equation}
	\label{26}
\left(\alpha^2 q^2-\gamma \alpha \beta^2\right) p^4+\left(2 \alpha q^2-\beta^2\right) p^2+q^2=0,
\end{equation}
where there are at most four possible values of $p$ such that (\ref{26}) holds. This finishes the proof that $\Sigma_{(\alpha, \beta,\gamma)}$ is either empty or finite. 
 
 Now, we assume the set $\Sigma_{(\alpha, \beta,\gamma)}$ is not empty, and suppose there exists two different paris $(p_0,q_0)$, $(p_1,q_1)$ $\in\Sigma_{(\alpha, \beta,\gamma)}$, where $p_0 \neq p_1$ or $q_0 \neq q_1$. We intend to show the uniqueness of element in  $\Sigma_{(\alpha, \beta,\gamma)}$ under conditions (i) and (ii) with contradictions.
\begin{itemize}
 	\item[$\bullet$]If $q_0 = q_1$, $p_0$ and $p_1$ are two roots of (\ref{26}), hence one has either $ p_0^2=p_1^2 $ or,
 	$$
 	p_0^2 \cdot p_1^2=\frac{q_0^2}{\alpha^2 q_0^2-r \alpha \beta^2}.
 	$$
Since $p_0, p_1\in\mathbb{N}$, for the case $ p_0^2=p_1^2 $, it leads to $p_0 = p_1$. And for the case,
$$
 p_1^2=\frac{q_0^2}{\left(\alpha^2 q_0^2-\gamma \alpha \beta^2\right) p_0^2},
$$
it is an integer which violates the assumptions.
\item[$\bullet$] If $q_0 \neq q_1$, one has,
$$
p^2=\frac{\beta^2-2 \alpha q_1^{2} \pm \sqrt{\left(2 \alpha q_1^{2}-\beta^2\right)^2-4\left(\alpha  q_1^{2}-\gamma \alpha \beta^2\right) q_1^{ 2}}}{2\left(\alpha^2 q_1^{2}-\gamma \alpha \beta^2\right)},
$$
and at least one of them is $p_1^{2}$ which is an integer, then this again contradicts with the assumptions.
\end{itemize}

The proof is now complete.
\end{proof}

We then study the norm of the operator $\mathcal{L}$.
\begin{lem}
\label{lem 2}
For $\alpha, \beta>0$, $0<\gamma<1$, there exists a constant $M>0$ (depending only on $(\alpha, \beta, \gamma))$ such that, for any pair $(p, q) \in \mathbb{N} \times \mathbb{Z}$, $(p,|q|) \notin \Sigma_{(\alpha, \beta,\gamma)}$ and $(p, q) \neq(0,0)$, we have
\begin{equation}
	\label{27}
	 C:=\frac{p\left(\beta p +\gamma \alpha \beta p^3 \right)+p|q|\left(1+\alpha p^2\right)}{\left|q^2\left(1+\alpha p^2\right)^2- \beta^2p^2\left(1+\gamma \alpha p^2\right)\right|} \leq M ,
\end{equation}
\begin{equation}
	\label{28}
	 D:=\frac{p|q|\left(1+\alpha p^2\right)+\beta p^2 }{\left|q^2\left(1+\alpha p^2\right)^2-\beta^2p^2 \left(1+\gamma \alpha p^2\right)\right|} \leq M .
\end{equation}
\end{lem}
\begin{proof}
 Since $\alpha, \beta>0$, $0<\gamma<1$, and $p \in\mathbb{N}$, we have,
\begin{equation}
	\label{29}
C=\frac{\beta p^2(\frac{1+\gamma\alpha  p^2}{1+\alpha p^2})+p|q|}{\left|q^2\left(1+\alpha p^2\right)- \beta^2p^2 (\frac{1+\gamma\alpha p^2}{1+\alpha p^2})\right|} \leq \frac{\beta p^2+p|q|}{\left|q^2\left(1+\alpha p^2\right)- \beta^2 p^2(\frac{1+\gamma\alpha p^2}{1+\alpha p^2})\right|} .
\end{equation}
Now, let us first consider the pairs $(p, q)$ satisfying $|q| \geqslant 2 \beta / \sqrt{\alpha}$, then one has,
\begin{equation}
\label{30}
q^2\left(1+\alpha p^2\right)-\beta^2p^2  \frac{1+\gamma\alpha p^2}{1+\alpha p^2}\geqslant q^2\left(1+\alpha p^2\right)- \beta^2p^2 \geqslant \frac{4 \beta^2}{\alpha}+ 3\beta^2p^2>0.
\end{equation}
Thus, this leads to,
$$
C\leqslant \frac{ \beta p^2+p|q|}{q^2(1+\alpha p^ 2)- \beta^2p^2}=\frac{p(\beta p+|q|)}{(|q| \sqrt{1+\alpha p^2}+\beta p)(|q| \sqrt{1+\alpha p^2}-\beta p) }\leqslant\frac{p}{|q| \sqrt{1+\alpha p^2}-\beta p }.
$$
Recall from the condition $|q| \geqslant 2 \beta / \sqrt{\alpha}$, then one can obatin,
$$
C\leqslant\frac{p}{\frac {2\beta }{\sqrt{\alpha}}(\sqrt{1+\alpha p^2})-\beta p }\leqslant\frac{p}{2\beta p}=\frac{1}{2\beta}.
$$
Next, we consider the case for $|q| < 2 \beta / \sqrt{\alpha}$, where the only possible values for $|q|$ are  $0,1,\ldots,K$ with $K=[2\beta / \sqrt{\alpha}]$. For a fixed value of $q$, we can define the sequence of number $\left\{d_p\right\}$ following from (\ref{29}),
$$
{d_p}:=\frac{\beta p^2 +p|q|}{\left|q^2\left(1+\alpha p^2\right) -   {\beta}^2p^2\frac{1+\gamma\alpha p^2}{1+\alpha p^2}\right|},~ p\in\mathbb{N}.
$$
As $ p $ tends to $+\infty$, the limit of  $\left\{d_p\right\}$ exists, i.e.
$$
\lim _{p \rightarrow +\infty}{d_p}=\lim _{p \rightarrow +\infty} \frac{\beta p^2 +p|q|}{\left|q^2\left(1+\alpha p^2\right) -   {\beta}^2p^2\frac{1+\gamma\alpha p^2}{1+\alpha p^2}\right|}=\frac{\beta}{\left|\alpha q^2-\gamma\beta^2\right|},
$$
where $q^2\neq\frac{\gamma\beta^2}{\alpha}$ since $(+\infty,\sqrt{\frac{\gamma\beta^2}{\alpha}}) \in \Sigma_{(\alpha, \beta,\gamma)}$.
Therefore, according to boundedness of the existence of sequence, we set,
$$
M_1=\sup\left\{d_p\right\},~ p\in\mathbb{N},~ q=0,1,\ldots,K,
$$
and take $M=\max \left\{\frac{1}{2 \beta}, M_1\right\}$, the proof of (\ref{27}) is complete.
We can simplify $D$ as follows,
$$
	D =\frac{p|q|+\beta p^2 \frac{1}{1+\alpha p^2}}{\left|q^2\left(1+\alpha p^2\right)-\beta^2 p^2 (\frac{1+\gamma \alpha p^2}{1+\alpha p^2})\right|} 
	\leq \frac{p|q|+ \beta p^2}{\left|q^2\left(1+\alpha p^2\right)-\beta^2 p^2(\frac{1+\gamma \alpha p^2}{1+\alpha p^2})\right|},
$$
the proof of (\ref{28}) is then similar to the one for (\ref{27}). Therefore, omitted.
\end{proof}
In the following, let us focus on the situation where $(p_0, q_0)$ is the only solution in $\Sigma_{(\alpha_0, \beta_0,\gamma_0)}$. Then, based on the Lemma~\ref{lem 2}, we have the following straightforward results of  the linearized operator $\mathcal{L}_0:=\mathcal{L}_{\left(\alpha_0, \beta_0,\gamma_0\right)}$.
\begin{prop}
	\label{prop 2.1}
Assuming generally that $(\alpha_0, \beta_0, \gamma_0)$ is in the situation where $(p_0, q_0)$ is the only solution in ~$\Sigma_{(\alpha_0, \beta_0,\gamma_0)}$. For $ (p,|q|)\ne(p_0, q_0)$ and  $ (p,q)\ne(0, 0)$, we have,
\begin{equation}
	\label{31}
	\begin{aligned}
	\left|\eta_{p q}\right|+\left|u_{p q}\right| \leqslant M(|f_{pq}|+|\tilde{g}_{p q}|),
	\end{aligned}
\end{equation}
which can give estimates for $\left(\eta_{p q}, u_{p q}\right)$ in terms of $\left(f_{p q}, \tilde{g}_{p q}\right)$. 
	\end{prop}

Now, we study the kernels of $\mathcal{L}_0$ in $H_{\natural\natural, 0}^{k, e}\times H_{\natural\natural}^{k, o}$. We can directly obtain the trivial kernel $\zeta_0=(1,0)$ of the linearized operator
$\mathcal{L}_0$, for the case when $(p,q)=(0,0)$. It remains to study the nontrivial kernel of the linearized operator
$\mathcal{L}_0$, corresponding to the case when $(p,|q|)=(p_0,q_0)$. The eigenvectors are defined as follows,
$$
\begin{aligned}
	&\xi_0:=\xi_{p_0,q_0}=\left(e^{i q_0 t} \cos p_0 x,-i \sqrt{1+\gamma_0 \alpha_0 p_0{ }^2} e^{iq_0t} \sin p_0 x\right),\\
	&\bar{\xi}_0:=\xi_{p_0,-q_0}=\left(e^{-i q_0 t} \cos p_0 x,i \sqrt{1+\gamma_0 \alpha_0 p_0{ }^2} e^{-iq_0t} \sin p_0 x\right),
\end{aligned}
$$
and it then leads to,
$$
\mathcal{L}_0 \xi_0=\mathcal{L}_0 \bar{\xi}_0=0 .
$$
With the Hermitian scalar product in $(L_{\natural\natural}^{2})^2$, the compatibility condition (\ref{23}) reads,
$$
\left\langle F, \xi_0\right\rangle=\left\langle F, \bar{\xi}_0\right\rangle=0,
$$
where $F=(\tilde{g}, f)\in H_{\natural\natural}^{k, e} \times H_{\natural\natural}^{k, o}$. Following the kernel defined above, we let $U=(\eta,u)$ is orthogonal to $\zeta_0$, $\xi_0$ and  $\bar{\xi}_0$ in $(L_{\natural\natural}^{2})^2$, where $U=(\eta,u)$ $\in$  $H_{\natural\natural, 0}^{k, e} \times H_{\natural\natural}^{k, o}$. Therefore, it ensures the  uniqueness of solution for $\mathcal{L}_0{U}=F$. Recall from (\ref{24}) and we have, 
\begin{equation}
	\label{32}
\begin{aligned}
	&\eta_{00}=0,\\
	& \eta_{p_0, \pm q_0}=\frac{\mp i \sqrt{1+\gamma_0 \alpha_0 p_0^2} p_0 f_{p_0, \pm q_0}}{q_0\left(1+\alpha_0 p_0^2\right) \sqrt{1+\gamma_0 \alpha_0 p_0^2}+\beta_0p_0 }, \\
	& u_{p_0, \pm q_0}=\frac{p_0 f_{p_0, \pm q_0}}{q_0\left(1+\alpha_0 p_0^2\right) \sqrt{1+\gamma_0 \alpha_0 p_0^2}+\beta_0p_0 }.
\end{aligned}
\end{equation}
Then, based on Proposition $\ref{prop 2.1}$, we are now able to define an bounded operator $\widetilde{\mathcal{L}_0}^{-1},$
$$
U=\widetilde{\mathcal{L}_0}^{-1}F,
$$
acting on $H_{\natural\natural}^{k, e} \times H_{\natural\natural}^{k, o}$, $k\geqslant0$, solving $\mathcal{L}_0 U=F$,
\begin{itemize}
\item[$(i)$] for $\Delta \neq 0, p > 0$ and $q \in \mathbb{Z}$,
$$
U_{p q}=(\eta_{pq},u_{pq})=\widetilde{\mathcal{L}_0}^{-1}F_{p q},
$$
with $\eta_{p q}$ and $u_{p q}$ given by (\ref{21}) and (\ref{22}), that is,
$$
	\begin{aligned}
		\eta_{p q}=-\Delta^{-1} p\left[i q\left(1+\alpha p^2\right) f_{p q}+ \beta p\tilde{g}_{p q}\right], \quad\quad\quad\quad\quad
	\end{aligned}
$$
$$
	\begin{aligned}
		u_{p q}=-\Delta^{-1} p\left[\left( \beta p+\gamma  \alpha \beta p^3\right) f_{p q }-i q\left(1+\alpha p^2\right) \tilde{g}_{p q}\right];
	\end{aligned}
$$
\item[$(ii)$] for $ (p,q)=(p_0, \pm q_0)$,
$$
U_{p_0,\pm q_0}=(\eta_{p_0,\pm q_0},u_{p_0,\pm q_0}),
$$
with $\eta_{p_0,\pm q_0}$ and $u_{p_0,\pm q_0}$ given by (\ref{32});
\item[$(iii)$] for $ (p,q)=(0,0)$,
$$
U_{0,0}=(0,0).
$$
\end{itemize} 

Let us introduce the $O(2)$ group invariance, corresponding to the invariance of the system under reflection $x \rightarrow -x$ and shifts in $t$. The symmetry operator is defined by,
$$
\{\mathcal{S}(\eta, u)\}(x, t)=(\eta(-x, t),-u(-x, t)),
$$
and for any $\tau$ real,  the linear operator $\mathcal{T}_\tau$ representing the shifts is defined by, 
$$
\left\{\mathcal{T}_\tau(\eta, u)\right\}(x, t)=(\eta(x, t+\tau), u(x, t+\tau)).
$$
One can obtain,
$$
\mathcal{S} \mathcal{T}_{-\tau}=\mathcal{T}_\tau \mathcal{S},\quad \mathcal{T}_{2 \pi }=\mathbb{I}.
$$
Hence the $O(2)$ symmetry is satisfied. Then, we have the following statements.
\begin{lem}
	\label{lem 2.3}
Assuming generally that $(\alpha_0, \beta_0, \gamma_0)$ is in the situation where $(p_0, q_0)$ is the only solution in $\Sigma_{(\alpha_0, \beta_0,\gamma_0)}$. Then, for any given,
$$
F=(\tilde{g},f)=(\left(\mathbb{I}-\pi_0\right) g,f)\in H_{\natural\natural}^{k, e} \times H_{\natural\natural}^{k, o},  k\geqslant0,
$$
satisfying the following compatibility conditions,
$$
\left\langle F, \xi_0\right\rangle=\left\langle F, \bar{\xi}_0\right\rangle=0,
$$
the general solution $U=(\eta, u) \in H_{\natural\natural, 0}^{k, e} \times H_{\natural\natural}^{k, o},  k\geqslant0$, of the system,
$$
\mathcal{L}_{_0} U=F,
$$	
is given by,
\begin{equation}
U=\widetilde{\mathcal{L}_0}^{-1}F+A \xi_0+\bar{A} \bar{\xi}_0+B \zeta_0,
\end{equation}
where,
$$
\zeta_0=(1,0), \quad \xi_0=\left(e^{i q_0 t} \cos p_0 x,-i \sqrt{1+\gamma_0 \alpha_0 p_0^2} e^{i q_0 t} \sin p_0 x\right), 
$$
A $\in \mathbb{C}$, $B\in \mathbb{R}$, and $\widetilde{\mathcal{L}_0}^{-1}$ is the bounded linear operator defined above. Moreover, we have the following properties,
\begin{equation}
\begin{aligned}
	\label{34}
\mathcal{S} \mathcal{L}_0 =\mathcal{L}_0 \mathcal{S},~\mathcal{S}\xi_0=\bar{\xi}_0,~\mathcal{S} \zeta_0=\zeta_0,~~~~\\
 \mathcal{T}_\tau \mathcal{L}_0 =\mathcal{L}_0\mathcal{T}_\tau ,\mathcal{T}_\tau \xi_0=\mathrm{e}^{\mathrm{i} q_0 \tau} \xi_0, \mathcal{T}_\tau \zeta_0=\zeta_0.
	\end{aligned}
\end{equation}
\end{lem}
\section{Bifurcating solutions of Bona-Smith system}
Let us consider system ($\ref{15}$) for parameter values $(\alpha, \beta,\gamma)=\left(\alpha_0+\mu, \beta_0+\nu,\gamma_0\right)$, we then can rewrite the system into following form,
$$
\left\{\begin{array}{l}
	u_t+\beta_0 \eta_x-\gamma_0 \alpha_0 \beta_0 \eta_{x x x}-\alpha_0 u_{x x t}=-\Big(     (\beta_0+\nu) \frac{u^2}{2}+\nu \eta-\big((\alpha_0+\mu)(\beta_0+\nu)\gamma_0-\alpha_0\beta_0\gamma_0\big) \eta_{x x}-\mu u_{xt}\Big)_x, \\
	\eta_t+\beta_0 u_x-\alpha_0 \eta_{x x t}=- \Big((\beta_0+\nu)u\eta+\nu u-\mu \eta_{x t}\Big)_x,
\end{array}\right.
$$
where $\left(\alpha_0, \beta_0, \gamma_0\right)$ is as in the above Theorem $\ref{lem 2.3}$ with $(\mu, \nu)$ close to 0. And let us look for nontrivial doubly periodic solutions $(\eta, u)$ in $H_{\natural\natural, 0}^{k, e} \times H_{\natural\natural}^{k, o}$. We write the full system as,
$$
\mathcal{L}_0 U=F,
$$
where $F=(\tilde{g}, f)=\left(\left(\mathbb{I}-\pi_0\right) g, f\right)$, and,
$$
\begin{aligned}
	&g=-\Big((\beta_0+\nu) \frac{u^2}{2}+\nu \eta-\big((\alpha_0+\mu)(\beta_0+\nu)\gamma_0-\alpha_0\beta_0\gamma_0\big) \eta_{x x}-\mu u_{xt}\Big),\\
	&f=-\Big((\beta_0+\nu)u\eta+\nu u-\mu \eta_{x t}\Big).
\end{aligned}
$$
For $k \geqslant 2$, using Sobolev imbedding theorem,  we can observe that,
$$
\left(u^2 / 2, u \eta\right) \in H_{\natural\natural}^{k, e} \times H_{\natural\natural}^{k, o}.
$$
Then $F=(\tilde{g}, f)$ has properties required in Theorem $\ref{lem 2.3}$, once compatibility conditions are satisfied. Now we can apply the Lyapunov-Schmidt method for finding the bifurcation equation.

Let $U=(\eta, u) \in H_{\natural\natural, 0}^{k, e} \times H_{\natural\natural}^{k, o}$, and rewrite the system as,
\begin{equation}
	\label{35}
\mathcal{L}_0 U+\nu \mathcal{J}U-\mu\mathcal{K} U_{x t}-((\alpha_0+\mu)(\beta_0+\nu)\gamma_0-\alpha_0\beta_0\gamma_0) \mathcal{G}U+\left(\beta_0+\nu\right) \mathcal{N}(U, U)=0,
\end{equation}
where,
$$
\mathcal{J} U=\left(\left(\mathbb{I}-\pi_0\right) \eta, u\right), \quad \mathcal{K} U_{x t}=\left(u_{x t}, \eta_{x t}\right), \quad \mathcal{G} U=(\eta_{x x}, 0),
$$
with the symmetric bilinear operator defined as follows, 
$$
2 \mathcal{N}\left(U_1, U_2\right)=\left(\left(\mathbb{I}-\pi_0\right) u_1 u_2, u_1 \eta_2+u_2 \eta_1\right),
$$
for $U_i=(\eta_i,u_i)$, $i=1,2$. We can obtain that ($\ref{35}$) is equivariant under the $O(2)$ symmetry defined above,
\begin{equation}
	\label{36}
	\begin{aligned}
		& \mathcal{T}_\tau \mathcal{J}=\mathcal{J} \mathcal{T}_\tau, \quad \mathcal{T}_\tau \mathcal{K}=\mathcal{K} \mathcal{T}_\tau,\quad \mathcal{T}_\tau \mathcal{G}=\mathcal{G} \mathcal{T}_\tau,\quad \mathcal{T}_\tau \mathcal{N}=\mathcal{N} \mathcal{T}_\tau, \\
		& \mathcal{S} \mathcal{J}=\mathcal{J} \mathcal{S},  \quad \mathcal{S} \mathcal{K}=\mathcal{K} \mathcal{S}, \quad \mathcal{S} \mathcal{G}=\mathcal{G} \mathcal{S},\quad \mathcal{S} \mathcal{N}=\mathcal{N} \mathcal{S} .
	\end{aligned}
\end{equation}
Then we decompose $U\in H_{\natural\natural, 0}^{k, e} \times H_{\natural\natural}^{k, o}$ as follows,
$$
U=(\eta,u)=\Theta+V,
$$
with,
$$ \Theta=A \xi_0+\bar{A} \bar{\xi}_0+B \zeta_0, \quad \left\langle V, \xi_0\right\rangle=\left\langle V, \bar{\xi}_0\right\rangle=\left\langle V, \zeta_0\right\rangle=0 ,
$$
where the inner product is the one of $(L_{\natural\natural}^2)^2, A \in \mathbb{C}$ and $B \in \mathbb{R}$ are constants. And as a consequence of the above decomposition, we have,
$$
	\left\langle U, \zeta_0\right\rangle=\left\langle\Theta+V, \zeta_0\right\rangle=\left\langle \Theta, \zeta_0\right\rangle 
	= 4 \pi^2B.
$$
Recall from $U=(\eta, u) \in H_{\natural\natural, 0}^{k, e} \times H_{\natural\natural}^{k, o}$, we can obtain, 
$$
\begin{aligned}
	\left\langle U, \zeta_0\right\rangle=\langle(\eta, u),(1,0)\rangle = \int_{-\pi}^\pi \int_{-\pi}^\pi \eta(x, t)~ \mathrm{d} x \mathrm{~d} t,
\end{aligned}
$$
hence $B$ is the average of $\eta(x, t)$, that is,
$$
B=\frac{1}{4 \pi^2} \int_{-\pi}^\pi \int_{-\pi}^\pi \eta(x, t)~ \mathrm{d} x \mathrm{~d} t .
$$
For $F=(\tilde{g}, f)\in H_{\natural\natural}^{k, e} \times H_{\natural\natural}^{k, o},$ let us define a projection $\mathcal{Q}_0$ as follows,
$$
	\mathcal{Q}_0 F =F - \frac{\left\langle F, \xi_0\right\rangle}{\left\langle \xi_0, \xi_0\right\rangle} \xi_0-\frac{\left\langle F, \bar{\xi}_0\right\rangle}{\left\langle\bar{\xi}_0, \bar{\xi}_0\right\rangle} \bar{\xi}_0,
$$
and we notice that  $\mathcal{Q}_0 F=F$ is a necessary condition for $F$ to belong to the range of $\mathcal{L}$. We observe that,
$$
\mathcal{J} \xi_0=\xi_0, \quad \mathcal{J} \zeta_0=0,
$$
hence $\mathcal{Q}_0 \mathcal{J} \Theta=0$.
Substituting $U$ in ($\ref{35}$) and applying the projection $\mathcal{Q}_0$  to (\ref{35}), we can obtain,
\begin{equation}
	\label{37}
\mathcal{L}_0 V+\nu\mathcal{Q}_0 \mathcal{J}V-\mu\mathcal{Q}_0 \mathcal{K} U_{x t}-((\alpha_0+\mu)(\beta_0+\nu)\gamma_0-\alpha_0\beta_0\gamma_0) \mathcal{G}U+\left(\beta_0+\nu\right)\mathcal{Q}_0  \mathcal{N}(U, U)=0,
\end{equation}
along with compatibility conditions,
\begin{equation}
	\label{38}
	\left\langle\nu\mathcal{J}V-\mu\mathcal{K} U_{x t}-((\alpha_0+\mu)(\beta_0+\nu)\gamma_0-\alpha_0\beta_0\gamma_0) \mathcal{G}U+\left(\beta_0+\nu\right)\mathcal{N}(U, U), \xi_0\right\rangle=0,
\end{equation}
\begin{equation}
	\label{39}
	\left\langle\underline{}\nu\mathcal{J}V-\mu\mathcal{K} U_{x t}-((\alpha_0+\mu)(\beta_0+\nu)\gamma_0-\alpha_0\beta_0\gamma_0) \mathcal{G}U+\left(\beta_0+\nu\right)\mathcal{N}(U, U), \bar{\xi}_0\right\rangle=0.
\end{equation}
Here, we can use the Lyapunov-Schmidt method to seek for a ``solution" from the ``projection"  ($\ref {37}$) and compatibility conditions ($\ref{38}$)-($\ref{39}$) will lead to relations between $A, B$ and $\mu,\nu$. For more details
of the Lyapunov-Schmidt method, the reader may refer to \cite{haragus2011local}.

From ($\ref{37}$), one has,
$$
V+{\widetilde{\mathcal{L}_0}}^{-1}\left\{\nu\mathcal{Q}_0 \mathcal{J}V-\mu\mathcal{Q}_0 \mathcal{K} U_{x t}-((\alpha_0+\mu)(\beta_0+\nu)\gamma_0-\alpha_0\beta_0\gamma_0)\mathcal{Q}_0  \mathcal{G}U+\left(\beta_0+\nu\right)\mathcal{Q}_0  \mathcal{N}(U, U)\right\}=0,
$$
which is of the form,
$$
\mathcal{F}(V, A, \bar{A},B, \mu, \nu)=0,
$$
and thanks to the boundedness properties of the operator $\widetilde{\mathcal{L}}_0^{-1}$, $\mathcal{F}$ is analytic:
$$
\left\{\left(H_{{\natural\natural}, 0}^{k, e} \times H_{\natural\natural}^{k, o}\right) \cap\left\{\xi_0, \zeta_0, \overline{\zeta_0}\right\}^{\perp}\right\} \times \mathbb{C}^2 \times \mathbb{R}^3 \rightarrow\left\{\left(H_{\natural\natural, 0}^{k, e} \times H_{\natural\natural}^{k, o}\right) \cap\left\{\xi_0, \zeta_0, \overline{\zeta_0}\right\}^{\perp}\right\}.
$$
Hence, equation ($\ref{37}$) may be solved in $ H_{\natural\natural, 0}^{k, e} \times H_{\natural\natural}^{k, o}$ with respect to $V$ by the implicit function theorem,
for $A, B, \mu, \nu$ close enough to 0 in $\mathbb{C} \times \mathbb{R}^4$. We shall uniquely obtain, 
$$
V=\mathcal{Y}( A, \bar{A}, B, \mu, \nu).
$$
Moreover because of the fact in (\ref{34}) that,
$$
\mathcal{S} \xi_0=\bar{\xi}_0,~ \mathcal{T}_\tau \xi_0=\mathrm{e}^{\mathrm{i}q_0 \tau} \xi_0,~
\mathcal{S} \zeta_0= \mathcal{T}_\tau \zeta_0=\zeta_0,
$$
and invariance properties shown in (\ref{36}), one has following symmetry properties, 
$$
\begin{aligned}
	& \mathcal{T}_\tau \mathcal{Y}(A,\bar{A},  B, \mu,\nu)=\mathcal{Y}\left(A \mathrm{e}^{\mathrm{i} q_0\tau}, \bar{A} \mathrm{e}^{-\mathrm{i} q_0 \tau},B ,\mu,\nu\right), \\
	& S \mathcal{Y}( A,\bar{A}, B ,\mu,\nu)=\mathcal{Y}(\bar{A}, A, B ,\mu,\nu).
\end{aligned}
$$
Due to above symmetry properties, we notice that if $A \equiv 0$, then $\Theta$ and $V$ are invariant under shifts $\mathcal{T}_\tau$. It results the existence of the family of ``trivial solutions" of ($\ref{35}$),  corresponding to $U=U^{(0)}=B \zeta_0$. 

Next, we intend to give the principal part of  $V$. For this purpose, we rewrite $V$  as the sum of terms with different orders,
$$
V=\textstyle\sum_n V_{\{|\vec{a}|=n\}},
$$
where $|\vec{a}|$ is the $l^1$-norm of $\vec{a}=(a_1, a_2,\cdots,a_5)$ for $a_k\in \mathbb{N}_0,~k=1,2, \cdots,5$.  $V_{\{|\vec{a}|=n\}}$ represents all $n$-th order terms in $V$, with 
$$
 V_{\{|\vec{a}|=n\}}=\textstyle\sum_i {A}^{a_1^{(i)}}{\bar{A}}^{a_2^{(i)}} B^{a_3^{(i)}} \mu^{a_4^{(i)}} \nu^{a_5^{(i)}} \vec{v}_i(x, t),~~\mbox{for } |\vec{a}^{(i)}|=n,~i\in \mathbb{N},
$$
where $\vec{v}_i(x,t)\in H_{{\natural\natural}}^{k}$.
It can be checked that there are no 0-th order or 1-th order terms in $V$. Hence, the principal part of $V$ shows that,
\begin{equation}
	\label{40}
\begin{aligned}
V(A, \bar{A}, B, \mu, \nu)&=V_2(A, \bar{A}, B, \mu, \nu )+\mathrm{O}\left\{(|A|+|B|+|\mu|+|\nu|)^3\right\},
\end{aligned}
\end{equation}
where $V_2$ represents terms of 2-th order and writes as,
$$
V_2=\mu\tilde{L}_0^{-1} \mathcal{Q}_0 \mathcal{K}{\Theta}_{x t}+\left(\gamma_0 \beta_0\mu+\gamma_0 \alpha_0 \nu\right) \tilde{L}_0^{-1}\mathcal{Q}_0 \mathcal{G}{\Theta}-\beta_0 \tilde{L}_0^{-1} \mathcal{Q}_0 \mathcal{N}(\Theta, \Theta).
$$
By calculations, we have,
$$
\begin{aligned}
	&\mathcal{K}{\Theta}_{x t}= p_0 q_0\left(A \sqrt{1+\gamma_0 \alpha_0 p_0^2} e^{i q_0t} \cos p_0 x+\bar{A} \sqrt{1+r_0 \alpha_0 p_0^2} e^{-i q_0 t} \cos p_0 x,-i A e^{i q_0 t} \sin p_0 x+i \bar{A} e^{-i q_0 t} \sin p_0 x\right), \\
	&\mathcal{G}{\Theta}=-p_0^2\left(A e^{i q_0 t} \cos p_0 x+\bar{A} e^{-i q_{0} t} \cos p_0 x, 0\right), \\
	&\mathcal{N}(\Theta, \Theta)=\left(\frac{1}{4}\left(1+\gamma_0 \alpha_0 p_0^2\right) \cos 2 p_0 x\left(A^2 e^{2iq_0t}-2|A|^2+\bar{A}^2 e^{-2 i q_ 0 t}\right)\right. \text {, } \\
	& \qquad\quad\quad\quad\left.-i \sqrt{1+\gamma_0 \alpha_0 p_0^2} \sin p_0 x\left(A B e^{i q_0 t}-A B e^{-i q_0  t}\right)-\frac{i}{2} \sqrt{1+\gamma_0 \alpha_0 p_0^2} \sin 2 p_0 x\left(A^2 e^{2 i q_0 t}-\bar{A}^2 e^{-2 i q_0 t}\right)\right), 
\end{aligned}
$$
and
$$
\begin{aligned}
	&\mathcal{Q}_0 \mathcal{K}{\Theta}_{x t}= \frac{r_0 \alpha_0 p_0^3 q_0}{2+\gamma_0 \alpha_0 p_0^2}\left( \sqrt{1+\gamma_0 \alpha_0 p_0^2} \cos p_0 x\left(A e^{i q_0 t}+\bar{A} e^{-i q_0 t}\right),i \sin p_0 x\left(A e^{i q_0 t}-\bar{A} e^{-i q_0 t}\right)\right),\\
	&\mathcal{Q}_0 \mathcal{G}{\Theta}=-\frac{p_0^2 \sqrt{1+\gamma_0 \alpha_0 p_0^2}}{2+\gamma_0 \alpha_0 p_0^2}\left( \sqrt{1+\gamma_0 \alpha_0 p_0^2} \cos p_0 x\left(A e^{i q_0 t}+\bar{A} e^{-i q_0 t}\right),i \sin p_0 x\left(A e^{i q_0 t}-\bar{A} e^{-i q_0 t}\right)\right), \\
	&\mathcal{Q}_0 \mathcal{N}(\Theta, \Theta)=\left(-\frac{1+\gamma_0 \alpha_0 p_0^2}{2+\gamma_0 \alpha_0 p_0^2} \cos p_0 x\left(AB e^{i q_0 t}+\bar{A}B  e^{-i q_0 t}\right)+\frac{1+\gamma_0 \alpha_0 p_0^2}{4} \cos 2 p_0 x\left(Ae^{i q_0 t}-\bar{A} e^{-i q_0 t}\right)^2\right., \\
	& \qquad\quad\quad\quad\quad\left.-\frac{i \sqrt{1+\gamma_0 \alpha_0 p_0^2}}{2+\gamma_0 \alpha_0 p_0^2} \cos p_0x\left(A B e^{i q_0 t}-\bar{A}B e^{-i q_0 t}\right)-\frac{i\sqrt{1+\gamma_0 \alpha_0 p_0^2}}{2}  \sin 2 p_0 x\left(A^2 e^{2 i q_0 t}-\bar{A}^2 e^{-2 i q_0 t}\right)\right) .
\end{aligned}
$$
Therefore, the principal part of $\mathcal{Y}$ is then given by,
$$
\begin{aligned}
	\mu \tilde{L}_0^{-1} \mathcal{Q}_0 \mathcal{K}{\Theta}_{x t}= & \frac{\mu \gamma_0 \alpha_0 p_0^4 q_0}{\left(q_0\left(1+\alpha_0 p_0^2\right) \sqrt{1+\gamma_0 \alpha_0 p_0^2}+\beta_0p_0 \right)\left(2+\gamma_0 \alpha_0 p_0^2\right)}\\&\left( \sqrt{1+\gamma_0 \alpha_0 p_0^2} \cos p_0 x\left(A e^{i q_0 t}+\bar{A} e^{-i q_0 t}\right),i \sin p_0 x\left(A e^{i q_0 t}-\bar{A} e^{-i q_0 t}\right)\right),~~~~~~~~~~~~~~~~~~~~~~~~~~~~~~~~
\end{aligned}
$$
$$
\begin{aligned}
 \left(\gamma_0 \beta_0 \mu+\gamma_0 \alpha_0 \nu\right) \tilde{L}_0^{-1} \mathcal{Q}_0 \mathcal{G}{\Theta}=& 
-\frac{\left(\gamma_0 \beta_0 \mu+\gamma_0 \alpha_0 \nu\right)p_0^3\sqrt{1+\gamma_0 \alpha_0 p_0^2}}{\left(q_0\left(1+\alpha_0 p_0^2\right) \sqrt{1+\gamma_0 \alpha_0 p_0^2}+\beta_0p_0 \right)\left(2+\gamma_0 \alpha_0 p_0 ^2\right)}\\&\left( \sqrt{1+\gamma_0 \alpha_0 p_0^2}\cos p_0 x\left(A e^{i q_0 t}+\bar{A} e^{-i q_0 t}\right), i \sin p_0 x\left(A e^{i q_0 t}-\bar{A} e^{-i q_0 t}\right)\right),~~~~~~~~~~~~
\end{aligned}
$$
and the part which is quadratic in $(A, \bar{A},B)$ is as follows,
$$
-\beta_0 \tilde{\mathcal{L}_0}^{-1} Q_0 \mathcal{N}(\Theta, \Theta):=\left(y^{(1)}, y^{(2)}\right),
$$
with,
$$
\begin{aligned}
	y^{(1)}=& \frac{\beta_0 p_0\left(1+\gamma_0 \alpha_0 p_0^2\right)}{\left(q_0\left(1+\alpha_0 p_0^2\right) \sqrt{1+\gamma_0 \alpha_0 p_0^2}+\beta_0p_0 \right)\left(2+\gamma_0 \alpha_0 p_0^2\right)} \cos p_0 x\left(A B e^{i q_0 t}+ \bar{A}B e^{-i q_0 t}\right)+\\& \frac{|A|^2}{2}\left(\frac{1+\gamma_0 \alpha_0 p_0^2}{1+4 \gamma_0 \alpha_0 p_0^2}\right) \cos 2 p_0 x+
	\alpha_1 \cos 2 p_0 x\left(A^2 e^{2 i q_0 t}+\bar{A}^2 e^{-2 i q_0 t}\right), & \\
	y^{(2)}=& \frac{i \beta_0 p_0 \sqrt{1+\gamma_0 \alpha_0 p_0^2}}{\left(q_0\left(1+\alpha_0 p_0^2\right) \sqrt{1+\gamma_0 \alpha_0 p_0^2}+\beta_0p_0 \right)\left(2+\gamma_0 \alpha_0 p_0^2\right)} \sin p_0 x\left(A B e^{i q_0 t}-\bar{A}B e^{-i q_0 t}\right) \\ &+
	i \beta_1 \sin 2 p_0 x\left(A^2 e^{2 i q_0 t}-\bar{A}^2 e^{-2 i q_0 t}\right),
\end{aligned}
$$
where, 
$$
\begin{aligned}
	& \alpha_1=\frac{\beta_0\left(1+3 \alpha_0 p_0{ }^2\right) (1+\gamma_0\alpha_0p_0^2)^{\frac{3}{2}}}{\left(4\alpha _ { 0 } p _ { 0 } q_0 \left(2-\gamma_0+5 \alpha_0 p_0^2+4 \gamma_0 \alpha_0^2p_0^ 4\right))\right.} ,\\
	& \beta_1=\frac{-\beta_0 \left(1+2 \alpha_0 p_0^2+3 \gamma_0 \alpha_0 p_0^2+4 \gamma_0 \alpha_0^2 p_0^4\right)(1+\gamma_0\alpha_0p_0^2)}{\left(4\alpha_0 p_0 q_0\left(2-\gamma_0+5 \alpha_0 p_0^2+4 \gamma_0 \alpha_0^2 p_0^4\right)\right)}.
\end{aligned}
$$
Now, substituting $V=\mathcal{Y}(A, \bar{A}, B, \mu, \nu)$ into ($\ref{38}$), we obtain the equation in $\mathbb{C}$ of the form,
$$
h(A, \bar{A}, B, \mu, \nu)=0,
$$
such that,
\begin{equation}
	\label{41}
	 h\left(A \mathrm{e}^{\mathrm{i} q_0 \tau}, \bar{A} \mathrm{e}^{-\mathrm{i} q_0 \tau}, B, \mu, \nu\right)=\mathrm{e}^{\mathrm{i} q_0 \tau} h(A, \bar{A},B, \mu, \nu),
\end{equation}
\begin{equation}
	\label{42}
	h(A, \bar{A}, B, \mu, \nu)=\bar{h}(\bar{A}, A, B, \mu, \nu),
\end{equation}
 for any real $\tau$,  while ($\ref{39}$) gives their complex conjugate. 
It results that $h$ takes the following form where the function $H$ is real valued,
\begin{equation}
	\label{43}
h(A, \bar{A},  B, \mu, \nu)=A H\left(|A|^2,  B, \mu, \nu \right),
\end{equation}
with the complex conjugate equation ($\ref{39}$) holds for same properties. 
Then, with the same order definition as $V$, we have
$$
h_2(A, \bar{A}, B, \mu, \nu)=	\left\langle\nu\mathcal{J}\Theta-\mu\mathcal{K} \Theta_{x t}-\left(\gamma_0 \beta_0\mu+\gamma_0 \alpha_0 \nu\right) \mathcal{G}\Theta+\beta_0\mathcal{N}(\Theta, \Theta), \xi_0\right\rangle.
$$
By calculations,
$$
\begin{aligned}
	& \left\langle \mathcal{J} \Theta, \xi_0\right\rangle=2 \pi^2\left(2+\gamma_0 \alpha_0 p_0^2\right) A, \\
	& \left\langle\mathcal{K} \Theta_{x t}, \xi_0\right\rangle=4 \pi^2 p_0 q_0 \sqrt{1+\gamma_0 \alpha_0 p_0{ }^2} A, \\
	& \left\langle\mathcal{G} \Theta, \xi_0\right\rangle=-2 \pi^2 p_0^2 A,\\
	& \left\langle\mathcal{N}(\Theta, \Theta), \xi_0\right\rangle=2 \pi^2\left(1+\gamma_0 \alpha_0 p_0{ }^2\right) A B,
\end{aligned}
$$
hence one has, 
$$
h_2=2 \pi^2A\left\{-\left(2 p_0 q_0 \sqrt{1+\gamma_0 \alpha_0 p_0^2}-\gamma_0 \beta_0 p_0^2\right) \mu+\left(2+2 \gamma_0 \alpha_0 p_0^2\right) \nu+\beta_0\left(1+\gamma_0 \alpha_0 p_0^2\right) B\right\}.
$$
And there is a 3-th order term $h_3^{'}:=2 \pi^2\beta_2A|A|^2$, which is a part of $h_3$. We observe that, 
$$
2 \mathcal{N}\left(U_1, U_2\right)=\left(\left(\mathbb{I}-\pi_0\right) u_1 u_2, u_1 \eta_2+u_2 \eta_1\right),
$$
and we have the principal part of ${V}$,  which is quadratic in $(A, \bar{A},B)$ under the form,
$$
y:=ABy^{(1,1)}+\bar{A}By^{(1,-1)}+|A|^2 y^{(2,0)}+A^2 y^{(2,2)}+\bar{A}^2 y^{(2,-2)}.
$$
Therefore the coefficient $\beta_2$ is as follows,
\begin{equation}
	\label{45}
	\begin{aligned}
		\beta_2 & =\frac{\beta_0}{2 \pi^2}\left\langle2 \mathcal{N}\left(y^{(2,0)} , \xi_0\right)+2 \mathcal{N}\left(y^{(2,2)}, \bar{\xi}_0\right) , \xi_0\right\rangle\\
		&  =\beta_0\left(-\frac{1}{4}\left(1+\gamma_0 \alpha_0 p_0^2\right)-\beta_1 \sqrt{1+\gamma_0 \alpha_0 p_0{ }^2}+\frac{\alpha_1}{2}\left(1+\gamma_0 \alpha_0 p_0^2\right)\right).
	\end{aligned}
\end{equation}
Based on $h_2$ and $h_3^{'}$, we can rewrite $h$ in the following form,
$$
\begin{aligned}
	h(A, \bar{A}, B, \mu, \nu)&=AH\left(|A|^2, B, \mu, \nu\right)\\&=2\pi^2A\bigg\{-\left(2 p_0 q_0 \sqrt{1+\gamma_0 \alpha_0 p_0^2}-\gamma_0 \beta_0 p_0^2\right) \mu+\left(2+2 \gamma_0 \alpha_0 p_0^2\right) \nu\bigg.\\&\bigg.~~~+\beta_0\left(1+\gamma_0 \alpha_0 p_0^2\right)B +\beta_2|A|^2+O\left(\left(|\mu|+|v|+|B|+|A|^2\right)^2\right) \bigg\}.
\end{aligned}
$$
The complex equation $h=0$ reduces to either $A=0$ or the real equation $H=0$.
Now noticing that an easy consequence is that A = 0, corresponding to the trivial family of solutions $U^{(0)}$ (already seen), another nontrivial solutions given by the solutions of the real equation,
$$
H\left(|A|^2, B, \mu, \nu\right)=0.
$$
At principal order this will arrive at (dropping the term of $\mathrm{O}\left(\left(|\mu|+|v|+|B|+|A|^2\right)^2\right)$), 
\begin{equation}
	\label{46}
|A|^2=-\frac{1}{\beta_2}\left\{-\left(2 p_0 q_0 \sqrt{1+\gamma_0 \alpha_0 p_0^2}-\gamma_0 \beta_0 p_0^2\right) \mu+\left(2+2 \gamma_0 \alpha_0 p_0^2\right) \nu+\beta_0\left(1+\gamma_0 \alpha_0 p_0^2\right) B\right\},
\end{equation}
for arbitrary $\mu, \nu, B$ close to 0, while the bifurcation only takes place for $\beta_2\neq0$ and the right of ($\ref{46}$) $\geqslant 0$. Finally we proved the following theorem.
\begin{lem}
	\label{lem 3.1}
Consider $\left(\alpha_0, \beta_0,\gamma_0\right)$ such that,
$$
\Sigma_{(\alpha, \beta, \gamma)}:=\left\{(p, q) \in N^2, q\left(1+\alpha p^2\right)- \beta p\sqrt{1+\gamma \alpha p^2}=0\right\},
$$
has a unique element $\left(p_0, q_0\right)$. Then, for $\mu, \nu, B$ close enough to 0, where $\alpha=\alpha_0+\mu, \beta=\beta_0+\nu, \gamma=\gamma_0, B $ is the average of $\eta(x, t)$, satisfying the right of ($\ref{46}$) $\geqslant 0$, there exists a family of bifurcation standing waves of the system ($\ref{15}$) in the form,
$$
U=\mathcal{T}_\tau U_0\in H_{\natural\natural,0}^{k, e} \times H_{\natural\natural}^{k, o}, ~k \geqslant 2,
$$
where $\tau$ corresponds to a time shift with respect to $U_0$, and,
$$
\begin{aligned}
	& U_0(x, t)=\left(\eta_0, u_0\right)(x, t), \\
	& \eta_0(x, t)=2|A| \cos q_0 t \cos p_0 x+B+\mathrm{O}\{|A|(|\mu|+|\nu|+|+|A|+|B|)\}, \\
	& u_0(x, t)=2 \sqrt{1+\gamma_0 \alpha_0 p_0^2}|A| \sin q_0 t \sin p_0 x+\mathrm{O}\{|A|(|\mu|+|\nu|+|A|+|B|)\}, \\
	& |A|^2=\frac{1}{\beta_2}\left\{\left(2+2 \gamma_0 \alpha_0 p_0^2\right) \mu-\left(2 p_0 q_0 \sqrt{1+\gamma_0 \alpha_0 p_0^2}-\gamma_0 \beta_0 p_0^2\right) \nu+\beta_0\left(1+\gamma_0 \alpha_0 p_0^2\right) B\right\},
\end{aligned}
$$
where $\beta_2(\neq 0)$ is given in ($\ref{45}$).
\end{lem}
\begin{cor}
	\label{cor 3.1}
With the assumption of the theorem above, the form of the free surface, even in x, with B = 0 is given by $\eta=\eta_0(x,t-\tau)$ with,
$$
\begin{aligned}
&~~~~\eta_0(x,t)\\&=2|A| \cos q_0 t \cos p_0 x+
2|A|\cos q_0 t \cos p_0 x\bigg\{\frac{\mu\left(\gamma_0 \alpha_0 p_0^4 q_0 \sqrt{1+\gamma_0 \alpha_0 p_0^2}-\gamma_0\beta_0p_0^3(1+\gamma_0 \alpha_0 p_0^2)\right)}{\left(q_0\left(1+\alpha_0 p_0^2\right) \sqrt{1+\gamma_0 \alpha_0 p_0^2}+\beta_0p_0 \right)\left(2+\gamma_0 \alpha_0 p_0^2\right)}\bigg.\\&\bigg.~~~-\frac{\nu\gamma_0\alpha _0p_0^3(1+\gamma_0 \alpha_0 p_0^2)}{\left(q_0\left(1+\alpha_0 p_0^2\right) \sqrt{1+\gamma_0 \alpha_0 p_0^2}+\beta_0p_0 \right)\left(2+\gamma_0 \alpha_0 p_0^2\right)}\bigg\}+ |A|^2\bigg\{\frac{1+\gamma_0 \alpha_0 p_0^2}{2\left(1+4 \gamma_0 \alpha_0 p_0^2\right)} cos 2 p_0 x\bigg.\\&\bigg.~~~+\frac{2\beta_0\left(1+3 \alpha_0 p_0^2\right)\left(1+\gamma_0 \alpha_0 p_0^2\right)^{\frac{3}{2}}}{\left(4 \alpha_0 p_0 q_0\left(2-\gamma_0+5 \alpha_0 p_0^2+4 \gamma_0 \alpha_0^2 p_0^4\right)\right)}\cos2q_0t\cos2p_0x\bigg\}+\mathrm{O}\bigg\{(|A|+|\mu|+|\nu|)^3\bigg\},
\end{aligned}
$$
for $\mu$, $\nu$close to 0, and $|A|$ being given in ($\ref{46}$).
\end{cor}
\begin{rem}
If we consider the higher-order terms in $V(A, \bar{A}, B, \mu, \nu)$ as follows, recalling from ($\ref{37}$),
$$
\begin{aligned}
V_3=&-\nu\tilde{\mathcal{L}}_0^{-1} \mathcal{Q}_0 \mathcal{J}{V_2}+\mu\tilde{\mathcal{L}}_0^{-1} \mathcal{Q}_0 \mathcal{K}{V_2}_{x t}+\left(\gamma_0 \beta_0\mu+\gamma_0 \alpha_0 \nu\right) \tilde{\mathcal{L}}_0^{-1}\mathcal{Q}_0 \mathcal{G}{V_2}+\gamma_0\mu\nu \tilde{\mathcal{L}}_0^{-1}\mathcal{Q}_0 \mathcal{G}{\Theta}\\&-\beta_0 \tilde{\mathcal{L}}_0^{-1} \mathcal{Q}_0 2\mathcal{N}(V_2, \Theta)-\nu\tilde{\mathcal{L}}_0^{-1} \mathcal{Q}_0 \mathcal{N}(\Theta, \Theta),~~~~~~~~~~~~~~~~~~~~~~~~~~~~~~~~~~~~~~~~~~~~~~~~~~~~~~~~~~~~~~~~~~~~~
\end{aligned}
$$
$$
\begin{aligned}
	V_4=&-\nu\tilde{\mathcal{L}}_0^{-1} \mathcal{Q}_0 \mathcal{J}{V_3}+\mu\tilde{L}_0^{-1} \mathcal{Q}_0 \mathcal{K}{V_3}_{x t}+\left(\gamma_0 \beta_0\mu+\gamma_0 \alpha_0 \nu\right) \tilde{\mathcal{L}}_0^{-1}\mathcal{Q}_0 \mathcal{G}{V_3}+\gamma_0\mu\nu\tilde{L}_0^{-1}\mathcal{Q}_0 \mathcal{G}{V_2}\\&-\beta_0 \tilde{\mathcal{L}}_0^{-1} \mathcal{Q}_0 2\mathcal{N}(V_3, \Theta)-\beta_0 \tilde{\mathcal{L}}_0^{-1} \mathcal{Q}_0 \mathcal{N}(V_2, V_2)-\nu\tilde{\mathcal{L}}_0^{-1} \mathcal{Q}_0 2\mathcal{N}(V_2, \Theta),~~~~~~~~~~~~~~~~~~~~~~~~~~~~~~~~~~~~~~~~
\end{aligned}
$$
$$
\begin{aligned}
	V_n=&-\nu\tilde{\mathcal{L}}_0^{-1} \mathcal{Q}_0 \mathcal{J}{V_{n-1}}+\mu\tilde{\mathcal{L}}_0^{-1} \mathcal{Q}_0 \mathcal{K}{V_{n-1}}_{x t}+\left(\gamma_0 \beta_0\mu+\gamma_0 \alpha_0 \nu\right) \tilde{\mathcal{L}}_0^{-1}\mathcal{Q}_0 \mathcal{G}{V_{n-1}}+\gamma_0\mu\nu \tilde{\mathcal{L}}_0^{-1}\mathcal{Q}_0 \mathcal{G}{V_{n-2}}\\&-\beta_0 \tilde{\mathcal{L}}_0^{-1} \mathcal{Q}_0 2\mathcal{N}(V_{n-1}, \Theta)-\nu\tilde{\mathcal{L}}_0^{-1} \mathcal{Q}_0 2\mathcal{N}(V_{n-2}, \Theta)\\&-
	\Sigma_{i,j\geqslant2}^{i+j=n}\beta_0
	 \tilde{\mathcal{L}}_0^{-1} \mathcal{Q}_0 \mathcal{N}(V_{i}, V_{j})-
	 \Sigma_{i,j\geqslant2}^{i+j=n-1}\nu
	 \tilde{\mathcal{L}}_0^{-1} \mathcal{Q}_0 \mathcal{N}(V_{i}, V_{j}),~n\geqslant5,~~~~~~~~~~~~~~~~~~~~~~~~~~~~~~~~~~~~~~~
\end{aligned}
$$
then the family of solutions of system ($\ref{15}$) which are more precise can be written as,
$$
U=A \xi_0+\bar{A} \bar{\xi}_0+B \zeta_0+\Sigma_{i=2}^{n}V_i+\mathrm{O}\left(\left(|A|+|B|+|\mu|+|\nu|\right)^{n+1}\right).
$$
Along with the higher-order terms in $h(A, \bar{A}, B, \mu, \nu)$, where $h=0$ represents the compatibility condition,
$$
\begin{aligned}
	h_3= & \left\langle\nu \mathcal{J} V_2-\mu \mathcal{K} V_{2 x t}-\left(\gamma_0 \beta_0 \mu+\gamma_0 \alpha_0 \nu\right) \mathcal{G} V_2-\gamma_0 \mu \nu \mathcal{G} \Theta+\beta_02\mathcal{N}(V_2, \Theta)+\nu\mathcal{N}(\Theta, \Theta), \xi_0\right\rangle,~~~~~~~~~~~~~~~~~~~~~~~~~~~~~~~~~~~~~~~~~~~~~~~~~~~~~~~~~~~~~~~~~~~~~~~
\end{aligned}
$$
$$
\begin{aligned}
	h_4= & \left\langle\nu \mathcal{J} V_3-\mu \mathcal{K} V_{3 x t}-\left(\gamma_0 \beta_0 \mu+\gamma_0 \alpha_0 \nu\right) \mathcal{G} V_3-\gamma_0 \mu \nu\mathcal{G} V_1+\beta_02\mathcal{N}(V_3, \Theta)+\beta_0\mathcal{N}(V_2, V_2)+\nu2\mathcal{N}(V_2, \Theta), \xi_0\right\rangle,~~~~~~~~~~~~~~~~~
\end{aligned}
$$
$$
\begin{aligned}
	h_n= & \left\langle\nu \mathcal{J} V_{n-1}-\mu \mathcal{K} V_{{n-1} x t}-\left(\gamma_0 \beta_0 \mu+\gamma_0 \alpha_0 \nu\right) \mathcal{G} V_{n-1}-\gamma_0 \mu \nu \mathcal{G}V_{n-2} \right. \\
	& \left.+\beta_02\mathcal{N}(V_{n-1}, \Theta)+\nu2\mathcal{N}(V_{n-2}, \Theta)+
	\Sigma_{i,j\geqslant2}^{i+j=n}\beta_0
 \mathcal{N}(V_{i}, V_{j})+	\Sigma_{i,j\geqslant2}^{i+j=n-1}\nu
 \mathcal{N}(V_{i}, V_{j}), \xi_0\right\rangle,~n\geqslant5,~~~~~~~~~~~~~~~~~~~~~~~~~~~~~~
\end{aligned}
$$
and therefore one has the bifurcation equation in the below,
$$
\begin{aligned}
	h(A, \bar{A}, B, \mu, \nu)= & A H\left(|A|^2, B, \mu, \nu\right) \\
	=&A\left\{\Sigma_{i=1}^{n-1} H_i\left(|A|^2, B, \mu, \nu\right)+O\left(\left(|\mu|+|\nu|+|B|+|A|^2\right)^n\right)\right\}\\=&0,
\end{aligned}
$$
where $H_i$ comes from $h_{i+1}$,~$i=1, \cdots, n-1$, and more precise $|A|^2$ is obtained when $H = 0$ via implicit function theorem.
\end{rem}
\section{Example}
In this section, we provide an example of the standing waves based on Theorem $\ref{lem 3.1}$ and Corollary $\ref{cor 3.1}$. We consider $\alpha_0= 5$, $\beta_0=4$ and $\gamma_0=\frac{1}{4}$, satisfying $\beta_2\neq 0$. It can be easily verified that there is only one element in $\Sigma_{(\alpha, \beta, \gamma)}$, that is, 
$$
(p_0,q_0)=(1,1).
$$

We next choose $(B, \mu, \nu)=(0,-0.1,-0.1)$ which satisfies the right of ($\ref{46}$) $\geqslant 0$, then, according to ($\ref{46}$), it follows,
$$
|A| \approx 0.572394.
$$

\begin{figure}[H]
	\centering
	\begin{minipage}{\textwidth}
		\centering
		\begin{subfigure}[t]{0.49\textwidth} 
			\includegraphics[width=\linewidth, keepaspectratio]{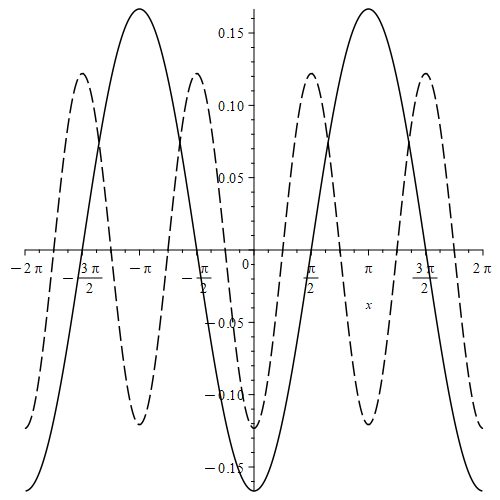} 
			\vspace{4pt}  
			\caption{\footnotesize The patterns showing in Corollary $\ref{cor 3.1}$ at $t=8$, the solid curve describes the terms of leading order for $\eta$ in $O(|A|)$ and the dash curve represents the terms purely in $\mathrm{O}\{|A|(|\mu|+|\nu|+|A|)\}$}
		\end{subfigure}
		\hfill
		\begin{subfigure}[t]{0.49\textwidth}
			\includegraphics[width=\linewidth, keepaspectratio]{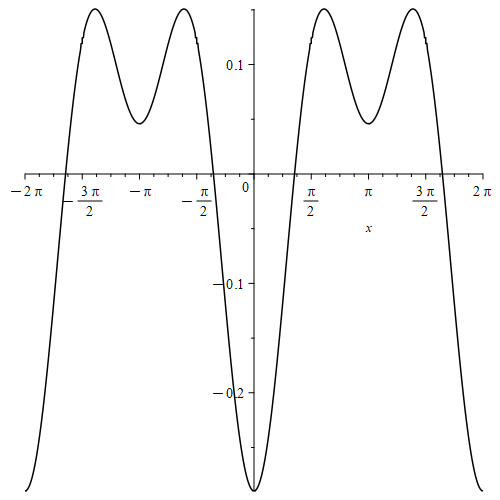}
			\vspace{4pt}
			\caption{\footnotesize The standing waves   given in the  Corollary at $t=8$.}
		\end{subfigure}
	\end{minipage}
	\caption{The the standing waves surface.}
	\label{fig1}
\end{figure} 
%

\smallskip

\centerline{\textbf{Acknowledgements}}

S. Li is supported by the National Natural Science Foundation of China (no. 12001084 and no. 12071061).
	\section*{Conflict of interest statement}
All authors declare that they have no conflicts of interest.

\section*{Data availability}
We confirm that no datasets are generated or analysed during the current study.

\end{document}